\documentclass[10pt,article,reqno]{amsart}
\usepackage{amssymb}
\usepackage{amsthm}
\usepackage{amsmath}
\usepackage{graphicx}
\usepackage{tikz}
\usepackage{tikz-3dplot}

\theoremstyle{plain}
\newtheorem{theorem}{Theorem}
\newtheorem{lem}{Lemma}
\newtheorem{cor}{Corollary}
\newtheorem{problem}{Problem}

\newcommand\Conv{\operatorname{Conv}}
\newcommand\Aff{\operatorname{Aff}}
\newcommand\ran{\operatorname{ran}}

\def\R{\mathbb{R}}
\def\N{\mathbb{N}}

\begin{document}

\title[Determining a point lying in a simplex]{Is it possible to determine a point lying in a simplex if we know the distances from the vertices?}
\author{Gy\"orgy P\'al Geh\'er}
\address{Bolyai Institute\\
University of Szeged\\
H-6720 Szeged, Aradi v\'ertan\'uk tere 1., Hungary}
\address{MTA-DE "Lend\"ulet" Functional Analysis Research Group, Institute of Mathematics\\
University of Debrecen\\
H-4010 Debrecen, P.O. Box 12, Hungary}
\email{gehergy@math.u-szeged.hu or gehergyuri@gmail.com}
\urladdr{http://www.math.u-szeged.hu/$\sim$gehergy/}
\keywords{Resolving set, bisector, simplex.}
\subjclass[2010]{Primary: 46C15 Secondary: 15A63, 47A30.}

\begin{abstract}
It is an elementary fact that if we fix an arbitrary set of $d+1$ affine independent points $\{p_0,\dots p_d\}$ in $\R^d$, then the Euclidean distances $\{|x-p_j|\}_{j=0}^d$ determine the point $x$ in $\R^d$ uniquely. 
In this paper we investigate a similar problem in general normed spaces which is motivated by this known fact. 
Namely, we characterize those, at least $d$-dimensional, real normed spaces $(X, \|\cdot\|)$ for which every set of $d+1$ affine independent points $\{p_0,\dots p_d\} \subset X$, the distances $\{\|x-p_j\|\}_{j=0}^d$ determine the point $x$ lying in the simplex $\Conv(\{p_0,\dots p_d\})$ uniquely. 
If $d=2$, then this condition is equivalent to strict convexity, but if $d > 2$, then surprisingly this holds only in inner product spaces.
The core of our proof is some previously known geometric properties of bisectors. 
The most important of these (Theorem \ref{Blaschke_thm}) is re-proven using the fundamental theorem of projective geometry.
\end{abstract}

\maketitle


\section{Introduction}

Let $X$ be a real normed space with norm $\|\cdot\|$, and $R,S \subseteq X$. 
We call $R$ a resolving set for $S$ if for any $s_1, s_2\in S$, the equations $\|r-s_1\| = \|r-s_2\| \; (r\in R)$ imply $s_1 = s_2$. 
We also say that $R$ resolves $S$, and in the literature this notion is also referred to as metric generator or determining set. 
This quite natural notion originally was defined for general metric spaces in \cite{RSbook} in 1953, but it attracted little attention at that time. 
In the theory of finite dimensional normed spaces (or Minkowski spaces) this notion was investigated in \cite{KaSt}. 
Namely, Kalisch and Straus proved that a $d$-dimensional normed space is Euclidean if and only if every subset $A$ which is not contained in a hyperplane is a resolving set for the whole space. 
We note that as a consequence of our results, we will strengthen this theorem in Corollary \ref{KaSt_cor}. 

In 1975, a very similar concept was introduced in graph theory (\cite{GraphEarly1,GraphEarly2}). 
Since then, several papers have been published concerning this direction (see e.g. \cite{Graph1,Graph2,Graph3,Graph4,Graph5,Graph6,Graph7} for some recent developments), and this topic has a wide range of applications for instance in informatics, robotics, biology and chemistry (see e.g. \cite{Appl1,Appl2,Appl3}). 

The notion of resolving sets in metric spaces is naturally related to the characterization of isometries of certain metric spaces. 
Recently, motivated by some problems in quantum mechanics, this notion was implicitly used in \cite{MoNa,MoTi,Na} in order to describe isometries of certain classes of matrices. 
Furthermore, with the help of resolving sets, the author of the present paper gave a new and elementary proof of a famous theorem of Wigner, which is very important in the foundation of quantum physics (see \cite{Ge}).

Also recently, motivated by complex analysis, some basic results on resolving sets in general metric spaces were provided in \cite{BB}. 
The metric dimensions of the hypercube in $\R^d$ and some important geometric spaces were discussed in \cite{Ba, HeMa}. 

In this paper we will consider general real normed spaces. 
For a number $d\in \N, d \geq 2$, we say that a normed space $X$ with $\dim X \geq d$ has the property \eqref{SRSd}, if it satisfies the following condition:
\begin{equation}\label{SRSd}\tag{SRS\textit{d}}
\emph{\text{every set of $d+1$ affine independent points resolves its convex hull.}}
\end{equation}
(Here SRS stands for ''simplex resolving set'').
In this paper we are interested in the problem of characterizing those normed spaces which satisfy the property \eqref{SRSd}. 
It will turn out that those at least two-dimensional normed spaces $X$ in which (SRS2) holds are precisely the strictly convex spaces (Theorem \ref{main2d_thm}). 
After that one would expect that we obtain the same conclusion if we consider \eqref{SRSd} with $d\geq 3$, since there is no immediate reason which suggests otherwise. 
But on the contrary, when $d\geq 3$, an at least $d$-dimensional space $X$ satisfies \eqref{SRSd} if and only if it is an inner product space (Theorem \ref{main>=3d_thm}). 

The characterization of strictly convex and inner product spaces is a classical field of functional analysis. 
There are several characterizations, and several of them were collected in the book of Amir \cite{Am} (see also e.g. \cite{recent3,Ma,recent1,recent2} concerning some recent results). 
We emphasize the well-known Jordan--von Neumann theorem which was proven originally for complex spaces in \cite{JovN}, but as was pointed out there, real spaces can be handled along the same lines (even with some simplifications). 
It states that a normed space $X$ is an inner product space if and only if its norm $\|\cdot\|$ satisfies the parallelogram identity:
\begin{equation}\label{par_eq}
2 (\|x\|^2 + \|y\|^2) = \|x+y\|^2 + \|x-y\|^2 \qquad (x,y\in X).
\end{equation}
This theorem further implies the following, which was also mentioned in \cite{JovN}: if $\dim X > 2$ ($\dim X > 3$, respectively) and the restriction of the norm to any (linear) subspace with dimension two (three, resp.) is Euclidean, then the norm of $X$ comes from an inner product as well. 
Obviously, a similar conclusion holds for strict convexity, which is straightforward from its definition. 
In other words, strict convexity and inner productness are two-dimensional properties. 
We also note that if one replaces the equality sign in \eqref{par_eq} by ''$\geq$'', then this inequality still characterizes inner productness (see \cite{Ma}).

From now on, if we do not say otherwise, $X$ will always denote a real normed space with norm $\|\cdot\|$. 
Whenever we say subspace, we will mean a linear subspace. 
On the other hand, when we consider a translated copy of a subspace, we will call it an affine subspace. 
In particular, we will say line if it is one-dimensional, plane if it is two-dimensional, or hyperplane if its codimension is one. 
We shall make use of the following notation: for every two points $x,y\in X, x\neq y$ let
\[
B(x,y) := \{z\in X\colon \|z-x\| = \|z-y\|\} \subseteq X,
\]
which is usually called the bisector of $x$ and $y$. 
Geometric properties of bisectors in finite dimensional normed spaces yield various deep characterizations of special normed spaces (see e.g. the survey papers \cite{MaSwWe,MaSw}). 
This notion is also naturally related to the study of Voronoi diagrams. 
We clearly have $B(x,y) = B(x-w,y-w) + w$ and $B(0,\lambda x) = \lambda \cdot B(0,x)$ $(w \in X, 0 \neq \lambda \in \R)$. 
It is well-known that this set is a hyperplane if $X$ is a $d$-dimensional Euclidean space. 
Furthermore, if $X$ is a strictly convex $d$-dimensional space, then all bisectors are homeomorphic images of hyperplanes (see \cite[Theorem 23]{MaSw} or \cite{GH}). 
We note that there exist non-strictly convex norms on $\R^d$ such that every bisector is homeomorphic to a hyperplane (see \cite[Example 3]{GH}). 

We state another useful characterization of inner product spaces in the next theorem which involves bisectors. 
This theorem could be obtained from the famous Blaschke--Marchaud theorem (\cite[Anhang VII.]{Bl} or \cite[(12.17)]{Am}), as it was noted in \cite[p. 123]{MaSw}. 

\begin{theorem}\label{Blaschke_thm}
Let $\|\cdot\|$ be a norm on $\R^d$ with $d\geq 3$. 
Every bisector $B(x,y)$ lies between two parallel hyperplanes if and only if the norm is Euclidean. 
\end{theorem}

We shall give a new proof of this theorem, using Faure's version of the fundamental theorem of projective geometry (see \cite{Fa}). 
We note that the above statement in two dimensions characterizes strictly convex spaces (see Lemma \ref{locstrictconv_lem} and \ref{nonstrictconv_lem}, later).

In the next two sections we will characterize normed spaces with property (SRS2), and \eqref{SRSd} with $d\geq 3$, respectively. 
We note that, by definition, a normed space $X$ with at least $d$ dimensions satisfies \eqref{SRSd} if and only if every $d$-dimensional subspace of $X$ fulfilles \eqref{SRSd} with the inherited norm. 
Let us explain briefly our approach. 
The property \eqref{SRSd} can be rephrased in the following way: 
there is no bisector $B(x,y)$ with $d+1$ affine independent points $p_0,p_1,\dots p_d \in B(x,y)$ such that we have $x,y\in\Conv(\{p_0,p_1,\dots p_d\})$. 
Here $\Conv(Z)$ denotes the convex hull of a given set $Z\subset X$. 
We will work with this latter version, which suggests the exploration of the geometric behaviour of bisectors in some detail. 

Finally, in Section 4, we will raise some problems which, in our opinion, are worth mentioning.


\section{A characterization of strict convexity}

The aim of this section is to prove the following theorem. 

\begin{theorem}\label{main2d_thm}
A real normed space $X$ of dimension at least two satisfies property \textup{(SRS2)} if and only if it is strictly convex.
\end{theorem}

First, we fix some notations which will be used throughout this paper. 
We will denote the line spanned by $x, y$ and the closed line segment between $x$ and $y$ by $\ell(x,y)$ and $[x,y]$, respectively. 
The open and closed unit balls of $X$ with respect to the norm $\|\cdot\|$ will be denoted by $B = B_{\|\cdot\|}$ and $\overline{B} = \overline{B}_{\|\cdot\|}$, respectively. 
The symbol $\partial\overline{B}$ will refer to the unit sphere. 
We will denote the affine hull of a given set $Z\subset X$ by $\Aff(Z)$. 

Before we prove Theorem \ref{main2d_thm}, we recall some lemmas concerning the behaviour of bisectors $B(x,y)$ in two dimensions. 
These lemmas will be used also in order to achieve our goal in the next section. 
In the next one, point (i) is a special case of \cite[Proposition 22]{MaSw}, and (ii) can be found as \cite[Lemma 2]{GH}.

\begin{lem}\label{locstrictconv_lem}
Let $\|\cdot\|$ be a norm on $\R^2$, and $x, y \in \R^2$ be two distinct points. 
Suppose that there exist exactly two points $p,-p \in \partial\overline{B}$ such that the supporting lines for $\overline{B}$ at these points are parallel to $\ell(x,y)$ (or equivalently, there is no non-degenerate segment on $\partial\overline{B}$ which is parallel to $\ell(x,y)$). 
Then the bisector $B(x,y)$ satisfies the following properties (see Figure 1):
\begin{itemize}
\item[\textup{(i)}] $B(x,y)$ is contained in the open region bounded by the lines $\ell(x,x+p)$ and $\ell(y,y+p)$, such that $p$ is not parallel to $y-x$;
\item[\textup{(ii)}] every line parallel to $\ell(x,y)$ intersects $B(x,y)$ in exactly one point.
\end{itemize} 
\end{lem}

\begin{figure}
\begin{tikzpicture}[scale=1.5]
	\draw [thick] plot [smooth cycle] coordinates {(1, 0) (0.840896, 0.840896) (0, 1) (-0.840896, 0.840896) (-1, 0) (-0.840896, -0.840896) (0, -1) (0.840896, -0.840896)};
	\draw [ultra thick] plot [smooth] coordinates {(-2,-0.720033) (-1.5,-0.348066) (-1,0.0477024) (-0.5,0.436767) (0,0.511878) (0.5,0.813354) (1,1.22221) (1.5,1.5994) (2,1.96415)};
	\draw[fill] (0,0) circle [radius=1pt];
	\node [below right, black] at (0,0) {$x$};
	\draw[fill] (-0.340482, 0.996623) circle [radius=1pt];
	\node [below right, black] at (-0.340482, 0.996623) {$y$};
	\draw[fill] (0.94,0.63) circle [radius=1pt];
	\node [below right, black] at (0.94,0.63) {$x+\|x-y\|\cdot p$};
	\draw [dashed] (-2,-1.33681) -- (2,1.33681);
	\draw [dashed] (-2,-0.112604) -- (2,2.56101);
\end{tikzpicture}
\caption{A circle of radius $\|x-y\|$ and centred at $x$, the bisector $B(x,y)$, and the region associated with it.}
\end{figure}
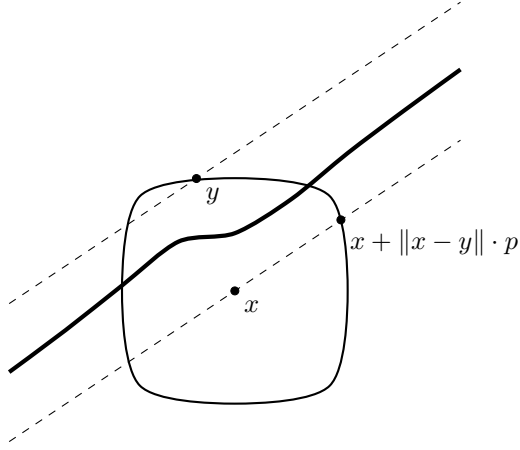

The following lemma is quite trivial, however we were not able to find it in the literature. 

\begin{lem}\label{nonstrictconv_lem}
Let $\|\cdot\|$ be a norm on $\R^2$, and $x, y \in \R^2$ be two distinct points. 
Suppose that there exists a non-degenerate segment on $\partial\overline{B}$ which is parallel to $\ell(x,y)$. 
Let $[a,b]$ be such a segment with the additional property that $a-b = \lambda (x-y)$ with some $\lambda > 0$. 
Then $B(x,y)$ satisfies the following two conditions:
\begin{itemize}
\item[\textup{(i)}] 
\begin{equation}\label{inftriang_eq}
\left\{ x + \tfrac{1}{\lambda}b + sa + tb \colon t,s\geq 0 \right\} \subseteq B(x,y)
\end{equation}
(see Figure 2);
\item[\textup{(ii)}] every line parallel to $\ell(x,y)$ intersects $B(x,y)$ in at least one point.
\end{itemize}
\end{lem}

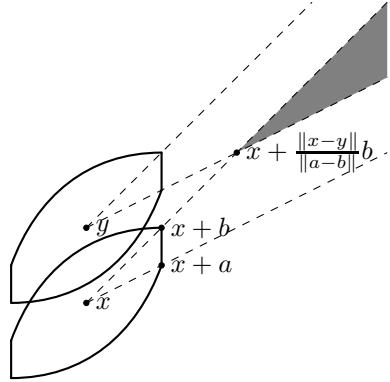
\begin{figure}
\begin{tikzpicture}
	\draw[thick] (0,0.5) to [out=70,in=180] (2,2);
	\draw[thick] (0,0) to [out=0,in=250] (2,1.5);
	\draw[thick] (0,0) -- (0,0.5); 
	\draw[thick] (2,1.5) -- (2,2);
		\draw[fill] (0,0) circle [radius=0.2pt];
		\draw[fill] (0,0.5) circle [radius=0.2pt];
		\draw[fill] (2,1.5) circle [radius=0.2pt];
		\draw[fill] (2,2) circle [radius=0.2pt];
	\draw[thin, dashed] (1,1) -- (3,3); 
	\draw[thin, dashed] (1,1) -- (5,3);
	\draw[fill] (1,1) circle [radius=1pt];
	\node [right, black] at (1,1) {$x$};
	\draw[thick] (0,1.5) to [out=70,in=180] (2,3);
	\draw[thick] (0,1) to [out=0,in=250] (2,2.5);
	\draw[thick] (0,1) -- (0,1.5); 
	\draw[thick] (2,2.5) -- (2,3);
		\draw[fill] (0,1) circle [radius=0.2pt];
		\draw[fill] (0,1.5) circle [radius=0.2pt];
		\draw[fill] (2,2.5) circle [radius=0.2pt];
		\draw[fill] (2,3) circle [radius=0.2pt];
	\draw[thin, dashed] (1,2) -- (4,5); 
	\draw[thin, dashed] (1,2) -- (3,3);
	\draw[fill] (1,2) circle [radius=1pt];
	\node [right, black] at (1,2) {$y$};
	\draw [black, dashed, thin, fill=gray] (5,5) -- (3,3) -- (5,4);
	\draw[fill] (2,1.5) circle [radius=1pt];
	\node [right, black] at (2,1.5) {$x+a$};
	\draw[fill] (2,2) circle [radius=1pt];
	\node [right, black] at (2,2) {$x+b$};
	\draw[fill] (3,3) circle [radius=1pt];
	\node [right, black] at (3,3) {$x+\tfrac{\|x-y\|}{\|a-b\|}b$};
\end{tikzpicture}
\caption{If the unit circle contains a non-degenerate segment, then the grey area is a subset of the bisector $B(x,y)$.}
\end{figure}

\begin{proof}
(i): Since $[a,b] \subset \partial\overline{B}$, we obtain that $\|t_1a + t_2b\| = t_1+t_2$ is fulfilled for every $t_1,t_2\geq 0$.
Therefore for every $s,t\geq 0$ we have
$$
\left\| \left(x + \tfrac{1}{\lambda}b + sa + tb\right) - x \right\| = \left\| sa + \left(\tfrac{1}{\lambda} + t\right)b \right\| = s + t + \tfrac{1}{\lambda}
$$
and 
$$
\left\| \left(x + \tfrac{1}{\lambda}b + sa + tb\right) - y \right\| = \left\| \tfrac{1}{\lambda}(a-b) + sa + \left(\tfrac{1}{\lambda} + t\right)b \right\| = \left\| \left(\tfrac{1}{\lambda} + s\right) a + t b \right\| = s + t + \tfrac{1}{\lambda}
$$
which verfies the statement.

(ii): Let $t\in\R$, and we define the following function:
\[
h\colon \ell(x+ta,y+ta) \to \R, \quad h(z) = \|x-z\|-\|y-z\|.
\]
We have $h(x+ta)\leq 0$ and $h(y+ta)\geq 0$. 
Since $h$ is continuous, we obtain $0\in\ran h$, and thus $B(x,y)\cap\ell(x+ta,y+ta) \neq \emptyset$ holds for every $t\in\R$. 
\end{proof}

Again, we were not able to find the following lemma in related publications, but we suspect that this might be known.

\begin{lem}\label{bisectline_lem}
Let $\|\cdot\|$ be a norm on $\R^2$. 
Suppose that $B(x,y)$ is a line for two distinct points $x, y \in \R^2$. 
Then the affine map $\phi$ for which every point of $B(x,y)$ is a fixpoint, and $\phi(x) = y$, is an isometry. 
Furthermore, if every bisector is a line, then the norm is Euclidean.
\end{lem}

\begin{proof}
Clearly, with $z = \tfrac{1}{2}(y-x)$ the bisector $B(-z,z)$ is a line $\R\cdot p$ going through the origin. 
We consider the linear transformation $\varphi\colon\R^2\to\R^2$ with $\varphi(z) = -z$ and $\varphi(p) = p$. 
We also choose an arbitrary point $\alpha p + \beta z \in \R^2$. 
Since $0\in B(\alpha p + \beta z, \alpha p - \beta z)$, we get $\|\varphi(\alpha p + \beta z)\| = \|\alpha p - \beta z\| = \|\alpha p + \beta z\|$, and we conclude that $\varphi$ is a linear isometry. Therefore the affine map $\phi$ defined in the statement is indeed an isometry. 

Concerning the second statement, we only have to use the characterization (2.8) in \cite{Am}.
\end{proof}

The first statement of the above lemma has a natural generalization in higher dimensions as well, and it can be handled along the same lines. 
However, we will not need it here. 
The second assertion in higher dimensions follows immediately from Theorem \ref{Blaschke_thm}. 
Now, we are in the position to give our proof of the first main result.

\begin{proof}[Proof of Theorem \ref{main2d_thm}]
Concerning the sufficiency part, let us assume that $X$ is strictly convex. 
Then every two-dimensional subspace is strictly convex as well, with respect to the inherited norm. 
Without loss of generality, we may restrict ourselves into one two-dimensional subspace, and consider two arbitrary distinct points $x,y\in X$. 
Then by Lemma \ref{locstrictconv_lem}, the bisector lies in an open region bounded by the parallel lines $\ell(x,x+p)$ and $\ell(y,y+p)$, and therefore the convex hull of the whole $B(x,y)$ is contained in that open region. 
Since $x$ and $y$ lies on the boundary of this region, they do not lie in the convex hull of any three points from $B(x,y)$. 
Therefore $X$ fulfilles (SRS2). 

Concerning the necessity part, let us assume that a non-degenerate segment $[a,b]$ lies on the boundary of $\overline{B}$. 
By Lemma \ref{nonstrictconv_lem}, we have $-\tfrac{1}{2}(a+b), (1+s)a+b, a+(1+s)b \in B(a,b)$ ($s> 0$) which are affine independent. 
It is easy to see that for large enough numbers $s$, the convex hull of these three points contains $a$ and $b$. 
This completes the proof.
\end{proof}

A statement about norms on $\R^d$ can be naturally transformed into another statement which considers centrally symmetric, convex, compact bodies with non-empty interior. 
This is done in the next corollary.

\begin{cor}\label{2_cor}
Let $m \geq 2$, and $K$ be a convex, compact body in $\R^m$ with non-empty interior such that $K = -K$. 
The following two conditions are equivalent:
\begin{itemize}
\item[\textup{(i)}] for every two numbers $\lambda_1,\lambda_2\in(0,\infty)$, and linearly independent vectors $v_1, v_2 \in \R^m$, the intersection 
\[
\Conv(\{0, v_1, v_2\}) \cap (\partial K) \cap (v_1 + \lambda_1 \cdot \partial K) \cap (v_2 + \lambda_2 \cdot \partial K)
\] 
has at most one element;
\item[\textup{(ii)}] $K$ is strictly convex. 
\end{itemize}
\end{cor}


\section{A characterization of inner product spaces}

This section is devoted to the verification of the following result.

\begin{theorem}\label{main>=3d_thm}
Let $d\geq 3$. 
A real normed space $X$ with $\dim X \geq d$ satisfies property \eqref{SRSd} if and only if it is an inner product space.
\end{theorem}

Before we present the proof of the above theorem, we will give our new proof of Theorem \ref{Blaschke_thm}, using the fundamental theorem of projective geometry. 

\begin{proof}[Proof of Theorem \ref{Blaschke_thm}]
One direction is trivial. 
Concerning the other one, we suppose that every bisector lies between two parallel hyperplanes. 
Throughout the proof we will use the natural inner product $\langle\cdot,\cdot\rangle$ and the (possibly non-Euclidean) norm $\|\cdot\|$ on $\R^d$. 
Let $x\neq y, x,y\in \R^d$. 
Lemmas \ref{locstrictconv_lem} and \ref{nonstrictconv_lem} imply that the hyperplanes, in between which $B(x,y)$ lies, do not contain any line which is parallel to $\ell(x,y)$.
In order to see this we assume the contrary, consider $\Aff(\{x,y,x+v\})$ where $v$ is a normal vector of the previously mentioned hyperplanes, and conclude the contradiction that $B(x,y) \cap \Aff(\{x,y,x+v\})$ is a bounded set. 
Therefore if we consider a three-dimensional subspace $Y \subseteq \R^d$ and two distinct points $y_1,y_2 \in Y$, we obtain that $B(y_1,y_2)\cap Y$ must lie between two parallel planes of $Y$.

We would like to conclude that the norm is Euclidean, and it is enough to show this for every three-dimensional subspace.
By the observations made above, we only have to deal with the $d = 3$ case of our theorem. 
Of course we only have to investigate bisectors of the form $B(0,x)$ where $\|x\| = 1$. 
We may also assume by Lemmas \ref{locstrictconv_lem} and \ref{nonstrictconv_lem} that the norm is strictly convex. 

Let $P_x$ be the set of those points $p \in \partial \overline{B}$ for which the supporting plane contains a line parallel to the vector $x$ 
(or in other words, for which the normal vector of the supporting plane is orthogonal to $x$). 
By Lemma \ref{locstrictconv_lem}, the set $P_x$ is the intersection of $\partial \overline{B}$ and a two-dimensional subspace which does not contain $x$.

Let $P(\R^3)$ be the Grassmann space of all one-dimensional subspaces 
(or alternatively, the projective space obtained by identifying antipodal points of the unit sphere $\mathbb{S}^2$). 
If $0\neq v\in\R^3$, then the subspace generated by $v$ will be denoted by $[v]$. 
Similarly, a subspace generated by some vectors $v_1,\dots v_m$ will be denoted by $[v_1,\dots v_m]$. 
We define the following transformation: 
\[
G := G_{\|\cdot\|} \colon P(\R^3) \to P(\R^3), \quad [v] \mapsto G([v]) \quad (v\neq 0),
\] 
where a normal vector of the supporting plane of $\overline{B}$ at the points $G([v])\cap(\partial\overline{B})$ is $v$. 
It is well-known that if $G$ is the identity transformation, then $B$ is a usual ball in $\R^3$ (the radius can be any positive number). 
Since $\overline{B}$ is centrally symmetric and strictly convex, the map $G$ is well-defined.
Obviously it is also surjective, but at this point nothing ensures its injectivity. 

Our observations about the set $P_x$ imply the following: 
if $v_1, v_2$ are linearly independent and $\{0\}\neq [v] \subseteq [v_1, v_2]$, then $G([v]) \in [G([v_1]), G([v_2])]$ is satisfied. 
This means that $G$ preserves projective lines. 
By Faure's version of the fundamental theorem of projective geometry (\cite[Theorem 3.1]{Fa}), and the well-known fact that the only endomorphism of the field $\R$ is the identity map, we obtain the existence of an injective linear map $A\colon \R^3\to\R^3$ such that $G$ is of the following form:
\begin{equation}\label{GA_eq}
G([v]) = [Av] \quad (v\neq 0)
\end{equation}
where $A$ is unique up to a non-zero scalar multiple.
It follows that $G$ is injective.  

Next, let $[v]^\perp$ be the two-dimensional subspace which is orthogonal to $v$ ($v\neq 0$). 
Let $A^T$ denote the transpose of $A$ with respect to the standard base in $\R^3$. 
For any injective linear transformation $C \colon \R^3\to\R^3$ we have 
\begin{equation}\label{C_eq}
C([v]^\perp) = \left[(C^T)^{-1} v\right]^\perp \qquad (v\neq 0)
\end{equation}
since $\langle v, w \rangle = \left\langle (C^T)^{-1} v, C w \right\rangle \, (v,w\in\R^3)$. 
Let us consider the linear image $C(\overline{B})$ of $\overline{B}$, which is the closed unit ball of the norm $|||\cdot||| = \|C^{-1}(\cdot)\|$, and the map $\widetilde{G} := G_{|||\cdot|||}$. 
If the supporting plane for $\overline{B}$ at $p\in\partial B$ is $p + [v]^\perp$ ($v\neq 0$), then the supporting plane for $C(\overline{B})$ at $Cp\in\partial C(B)$ is $Cp + C([v]^\perp)$. 
By \eqref{C_eq}, one normal vector of $Cp + C([v]^\perp)$ is $(C^T)^{-1} v$. 
Therefore we have
\[
\widetilde{G}([(C^T)^{-1} v]) = C (G([v])) = C [Av] = [CAv] \qquad (v\neq 0),
\]
which implies 
\begin{equation}\label{CACT_eq}
\widetilde{G}([v]) = [CAC^T v] \qquad (v\neq 0).
\end{equation}
If $A$ was negative or positive definite, then choosing $C = |A|^{1/2}$ would imply that $\widetilde{G}$ is the identity map, and therefore $B$ is an ellipsoid. 
Ou aim is to show that this is true. 

Let us assume that $M$ is a two-dimensional invariant subspace of $A$ (which exists by Jordan's decomposition). 
We observe that one normal vector of the supporting line for $\overline{B} \cap M$ at the points of $[Av] \cap \partial{\overline{B}}$ is $v$. 
It is quite easy to show, from the strict convexity of $\overline{B}$ and the bijectivity of $A$, that at each point of $\partial\overline{B} \cap M$ there is a unique supporting line. 
Therefore $\partial\overline{B} \cap M$ is smooth. 
Let $\{e_1,e_2\}$ be an orthonormal base of $M$. 
We set $v(t) = \cos t e_1 + \sin t e_2$ $(t\in\R)$, and consider the following parametrization of $\partial{\overline{B}} \cap M$: 
\[
\gamma\colon \R \to M, \quad \gamma(t) = c(t) \cdot Av(t),
\]
where $c\colon \R\to (0,\infty)$ is continuous and $\pi$-periodic. 
From the smoothness of $\partial\overline{B} \cap M$ and the definition of $\gamma$, the differentiability of $\gamma$ follows. 
If we write $\gamma$ with respect to the coordinates $\{Ae_1, Ae_2\}$, and use the fact that at every $t$ either $\tfrac{1}{\sin t}$ or $\tfrac{1}{\cos t}$ is differentiable, then we obtain the differentiability of $c$. 
We define the linear transformation $R\colon M \to M$ with $R e_1 = e_2$ and $R e_2 = -e_1$.
The following calculation is valid:
\[
0 = \langle \gamma'(t), v(t) \rangle = c'(t) \cdot \langle A v(t), v(t) \rangle + c(t) \cdot \langle A R v(t), v(t) \rangle \quad (t \in \R).
\]
Clearly, there exists at least two numbers $t_1, t_2 \in [0,\pi)$, $t_1 \neq t_2$ such that $c'(t_1) = c'(t_2) = 0$. 
From the above equation, we infer that $R v(t_1)$ and $R v(t_2)$ are two linearly independent eigenvectors of $A$. 
Since $M$ was an arbitrary two-dimensional invariant subspace, a straightforward application of Jordan's decomposition theorem gives that $A$ is diagonalizable. 

Now, when we chose the orthonormal base $\{e_1,e_2\}$ of $M$ in the above paragraph, we could have chosen $e_1$ to be an eigenvector associated with some eigenvalue $\lambda_1 (\neq 0)$. If $\lambda_2 (\neq 0)$ is the other eigenvalue of $A|M$ (which is possibly equal to $\lambda_1$), then a straightforward calculation gives us $A e_2 = \lambda_2 e_2 + b e_1$ with some $b \in \R$. 
We calculate the following:
\[
\langle A(\alpha e_1 + e_2), \alpha e_1 + e_2 \rangle = \lambda_1 \alpha^2 + b \alpha + \lambda_2 \qquad (\alpha \in \R).
\] 
If we have $\lambda_1\lambda_2 < 0$, then this can be zero for some $\alpha\in\R$. 
This is a contradiction by the very definition of $G$. 
Since every pair of eigenvalues of $A$ is either positive or negative, we may assume that all of them are positive (if they are negative, we simply consider $-A$ instead of $A$). 

Let $e_3\in\R^3$ be a unit vector which is orthogonal to $M$. 
We can choose an injective linear transformation $C\colon \R^3\to\R^3$ such that the vectors $e_j$ are all eigenvectors, and the matrix of $CAC^T|M$ represented in $\{e_1,e_2\}$ is of the form $\left(\begin{matrix}
1 & a \\
0 & 1
\end{matrix}\right)$.
At this point we note that $a=0$ is valid if and only if $b=0$ holds which is equivalent to the positive definiteness of $A|M$. 
We consider the following parametrization of $C(\partial{\overline{B}}) \cap M$: 
\[
\delta\colon \R \to C(\partial{\overline{B}}) \cap M, \quad \delta(t) = d(t) \cdot C A C^T v(t), 
\]
with some $\pi$-periodic function $d\colon \R\to (0,\infty)$, which is also differentiable. 
By \eqref{CACT_eq}, the following equation is satisfied for every $t\in\R$:
\[
0 = \langle \delta'(t), v(t) \rangle = d'(t) \cdot \langle C A C^T v(t), v(t) \rangle + d(t) \cdot \langle C A C^T R v(t), v(t) \rangle 
\]
\[
 = d'(t)\cdot\left[1+\tfrac{a}{2} \sin (2t)\right] + d(t) a \cos^2 t. 
\]
Obviously, $1+\tfrac{a}{2} \sin (2t) = 0$ cannot happen. 
But then we obtain
\[
(\log d(t))' = \frac{d'(t)}{d(t)} = - \frac{a \cos^2 t}{1+\tfrac{a}{2} \sin (2t)},
\]
which is either positive or negative everywhere, if we have $a\neq 0$. 
Hence in this case the function $d$ cannot be $\pi$ periodic, which is a contradiction. 
Therefore $a=0$, which implies that $A|M$ is positive or negative definite. 
Finally, since this is true for every two-dimensional invariant subspace of $A$, an easy application of Jordan's theorem verifies that $A$ is either negative or positive definite. 
\end{proof}

We note that our method could be used in order to obtain Blaschke's original characterization of ellipsoids. 
We state one immediate consequence of Theorem \ref{Blaschke_thm}.

\begin{cor}
The norm $\|\cdot\|$ on $\R^d$ $(d\geq 2)$ is Euclidean if and only if every bisector is a hyperplane.
\end{cor}

\begin{proof}
On one hand, an easy application of Theorem \ref{Blaschke_thm} verifies our statement if $d\geq 3$. 
On the other hand, the two-dimensional case was handled in Lemma \ref{bisectline_lem}.
\end{proof}

In order to achieve our goal, we will need the following technical lemma.

\begin{lem}\label{convfunct_lem}
Let $f\colon \R^2\to\R$ be a homogeneous function, i.e. we have 
\begin{equation}\label{hom_eq}
f(\lambda x_1, \lambda x_2) = \lambda\cdot f(x_1,x_2) \qquad (\forall\, x_1,x_2,\lambda\in\R).
\end{equation}
Suppose that if we restrict $f$ into any line segment, then this restriction is either convex or concave (this might be different on different segments). 
Then $f$ is necessarily linear.
\end{lem}

\begin{proof}
It is an elementary observation that on every line $\ell$, our function $f$ is either convex or concave. 
Furthermore, if $\ell$ does not go through the origin, and $f$ is convex (concave, respectively) on $\ell$, then by \eqref{hom_eq}, $f$ is convex (concave, resp.) on the open halfplane $(0,\infty)\cdot\ell$. 

Now, let $\ell_1$ and $\ell_2$ be two non-parallel lines in $\R^2$ such that none of them goes through the origin. 
We define $\ell_j^0$ to be the one-dimensional subspace parallel to $\ell_j$ ($j=1,2$). 
The subspace $\ell_j^0$ splits $\R^2$ into two open half planes $U_j^+ := (0,\infty) \cdot \ell_j$ and $U_j^- := (-\infty,0) \cdot \ell_j$. 
We may suppose without loss of generality that $f$ is convex on $\ell_1$, and then of course $f$ is concave on $-\ell_1$. 
Since $f$ is convex on $U_1^+$ and concave on $U_1^-$, we immediately obtain that $f$ has to be linear on one of the half lines $\ell_2\cap U_1^+$ or $\ell_2\cap U_1^-$. 
Without loss of generality, we may assume that $f$ is zero on this half line. 
Since $f$ is homogeneous, we obtain that $f$ is zero on one of the quarter planes $U_1^+\cap U_2^+$ or $U_1^-\cap U_2^+$. 
We denote this domain by $Q$. 
Of course $f$ is also zero on $-Q$. 

Let us consider a third line $\ell_3$, not parallel to the previously fixed ones, not going through the origin, and in addition we demand that $\ell_3\cap Q$ and $\ell_3\cap (-Q)$ are unbounded sets. 
Then $f$ is zero on both $\ell_3\cap Q$ and $\ell_3\cap (-Q)$. 
But this further implies that it must be zero on the whole of $\ell_3$ as well. 
Finally, linearity of our original $f$ on $\R^2$ follows quite easily.
\end{proof}

Now, we are in the position to verify our second main result.

\begin{proof}[Proof of Theorem \ref{main>=3d_thm}]
First, let us suppose that $X$ is not strictly convex. 
Then there exists a two-dimensional subspace $Y$ which is not strictly convex as well. 
By Theorem \ref{main2d_thm}, there are two distinct points $x,y\in Y$ such that we can choose three affine independent points $p_0,p_1,p_2 \in B(x,y)\cap Y$ which satisfy $x,y\in \Conv(\{p_0,p_1,p_2\})$. 
Since, by Lemmas \ref{locstrictconv_lem} and \ref{nonstrictconv_lem}, every line parallel to $\ell(x,y)$ intersects $B(x,y)$ in at least one point, we can choose points $p_3,\dots p_d \in B(x,y)$ such that $p_0,p_1,\dots p_d$ are still affine independent. 
Obviously, the convex hull of these $d+1$ affine independent points contains $x$ and $y$, and thus $X$ does not satisfy property \eqref{SRSd}.

Next, we suppose that $X$ is an inner product space. 
If we have two distinct points $x,y\in X$ and $d+1$ affine independent points $p_0,p_1,\dots p_d \in B(x,y)$ which satisfy $x,y\in\Conv(\{p_0,p_1,\dots p_d\})$, then we must also have $x,y\in \Aff(\{p_0,p_1,\dots p_d\})$. 
But $\Aff(\{p_0,p_1,\dots p_d\}) \cap B(x,y)$ is a $(d-1)$-dimensional affine subspace, and thus it cannot contain $p_0,p_1,\dots p_d$ all together. 
Therefore every inner product space fulfilles property \eqref{SRSd}. 

The only case which was left is when $X$ is strictly convex, but not an inner product space. 
First, we will deal with the case when $d=3$. 
Clearly, it is enough to handle the $\dim X = 3$ case, so we may assume in the sequel that this is satisfied. 
By Theorem \ref{Blaschke_thm} there exists a point $0\neq z\in X$ such that $B(0,z)$ does not lie between two parallel planes. 
We will show that in this case there exist four affine independent points in $B(0,z)$ such that their convex hull contains $0$ and $z$. 

From now on, we will use only linear space properties of $X$. 
By Lemma \ref{locstrictconv_lem}, for every two-dimensional subspace $P$ which contains $z$, the set $B(0,z)\cap P$ is contained in an open region of $P$ bounded by two parallel lines, one of them is $\R\cdot p_P$. 
We consider the set $G$ of all of these lines. 
Let $M$ be any two-dimensional subspace not containing $z$. 
If we consider it as the $(x,y)$ coordinate-plane, and the line $\ell(0,z)$ as the $z$ coordinate-axis, then by Lemma \ref{locstrictconv_lem} the set $G$ can be considered as a graph of a function $f$ which is homogeneous. 
Since $f$ was assumed to be non-linear, by Lemma \ref{convfunct_lem} we can choose $M$, and the $x$- and $y$-axes on it such that they satisfy the following conditions: 
$f(-1,0) = f(0,-1) = 0$, $z = (0,0,1)$, and there are numbers $0 < x_1 < x_2$, $0 < y_2 < y_1$ such that we have either $\alpha_1 := f(x_1,y_1) > 0$ and $\alpha_2 := -f(x_2,y_2) > 0$, or $\alpha_1 < 0$ and $\alpha_2 < 0$ (see Figure 3). 
We will only handle the first case, because the second one is quite the same. 

\begin{figure}
\tdplotsetmaincoords{75}{70}
\begin{tikzpicture}[tdplot_main_coords]
	\draw[thick,->] (-5,0,0) -- (5,0,0) node[anchor=north east]{$x$};
	\draw[thick,->] (0,-5,0) -- (0,5,0) node[anchor=north west]{$y$};
	\draw[thick,->] (0,0,-2) -- (0,0,3) node[anchor=south]{$z$};
	\draw[thick, gray, dashed] (-5,0,1) -- (0,0,1);
	\draw[thick, gray, dashed] (0,-5,1) -- (0,0,1);
	\draw [ultra thin, lightgray, fill=lightgray, opacity=0.4] (0,0,0) -- (0,0,0.98) -- (-5,0,0.98) -- (-5,0,0);
	\draw [ultra thin, lightgray, fill=lightgray, opacity=0.4] (0,0,0) -- (0,0,0.98) -- (0,-5,0.98) -- (0,-5,0);
	\draw [thick] plot [smooth] coordinates {(-5,0,0.6) (-4,0,0.7) (-3,0,0.3) (-2,0,0.4) (-1,0,0.6) (0,0,0.5)};
	\draw [thick] plot [smooth] coordinates {(0,-5,0.6) (0,-4,0.7) (0,-3,0.3) (0,-2,0.4) (0,-1,0.6) (0,0,0.5)};
	\draw[thick, gray, dashed] (0,0,0) -- (2.8,4,1.44);
	\draw[thick, gray, dashed] (0,0,1) -- (2.8,4,2.44);
	\draw [ultra thin, lightgray, fill=lightgray, opacity=0.4] (0,0,0.02) -- (0,0,0.98) -- (2.8,4,2.42) -- (2.8,4,1.46);
	\draw [thick] plot [smooth] coordinates {(0,0,0.5) (0.7,1,0.5) (1.4,2,1.3) (2.1,3,1.8) (2.8,4,2)};
	\draw[thick, gray, dashed] (0,0,0) -- (4,2.8,-1.6);
	\draw[thick, gray, dashed] (0,0,1) -- (4,2.8,-0.6);
	\draw [ultra thin, lightgray, fill=lightgray, opacity=0.4] (0,0,0.02) -- (0,0,0.98) -- (4,2.8,-0.62) -- (4,2.8,-1.58);
	\draw [thick] plot [smooth] coordinates {(0,0,0.5) (1,0.7,-0.1) (2,1.4,-0.2) (3,2.1,-0.8) (4,2.8,-1.2)};
	\draw [ultra thin, gray, fill=lightgray, opacity=0.2] (0,0,0) -- (5,0,0) -- (5,5,0) -- (0,5,0);
	\draw[thin, black, dotted] (2.8,4,2.44) -- (2.8,4,0);
		\draw[fill] (2.8,4,0) circle [radius=1pt];
		\node [right, black] at (2.8,4,0) {$\tilde q_1^t$};
	\draw[thin, black, dotted] (4,2.8,-1.58) -- (4,2.8,0);
		\draw[fill] (4,2.8,0) circle [radius=1pt];
		\node [right, black] at (4,2.8,0) {$\tilde q_2^t$};
	\draw[fill] (0,0,1) circle [radius=1pt];
	\node [above right, black] at (0,0,1) {$(0,0,1)$};
	\draw[fill] (2.8,4,1.44) circle [radius=1pt];
	\node [right, black] at (2.8,4,1.44) {$q_1^t$};
	\draw[fill] (2.8,4,2) circle [radius=1pt];
	\node [right, black] at (2.8,4,2) {$p_1^t$};
	\draw[fill] (4,2.8,-1.6) circle [radius=1pt];
	\node [right, black] at (4,2.8,-1.7) {$q_2^t$};
	\draw[fill] (4,2.8,-1.2) circle [radius=1pt];
	\node [right, black] at (4,2.8,-1.1) {$p_2^t$};
	\draw[fill] (-5,0,0) circle [radius=1pt];
	\node [left, black] at (-5,0,0) {$q_3^t$};
	\draw[fill] (-5,0,0.6) circle [radius=1pt];
	\node [left, black] at (-5,0,0.6) {$p_3^t$};
	\draw[fill] (0,-5,0) circle [radius=1pt];
	\node [above left, black] at (0,-5,0) {$q_4^t$};
	\draw[fill] (0,-5,0.6) circle [radius=1pt];
	\node [above left, black] at (0,-5,0.6) {$p_4^t$};
\end{tikzpicture}
\caption{Four two-dimensional half-region slices of the bisector $B((0,0,0),(0,0,1))$. 
Two of them go above the negative halves of the $x$- and $y$-axes, and the other two are placed above and below the positive quarter $(x,y)$-plane (lighter gray area).}
\end{figure}
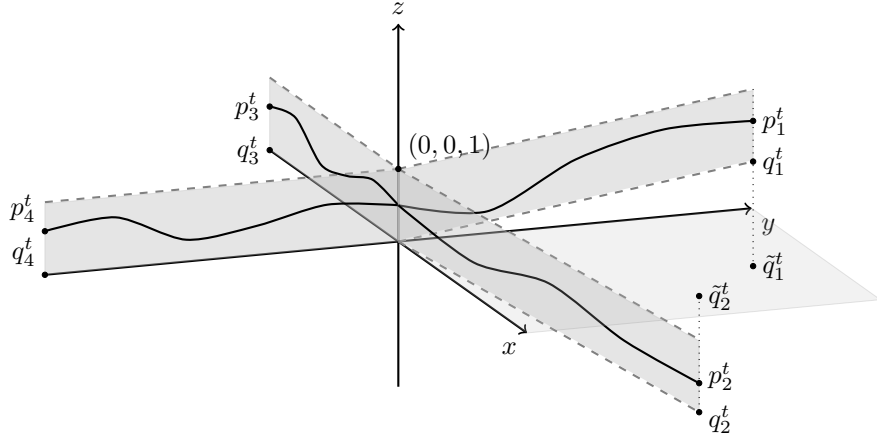

Now, we consider the following points:
\[
q_1^{t} = (tx_1, ty_1, t\alpha_1), \quad q_2^{t} = (tx_2, ty_2, -t\alpha_2), 
\]
\[
q_3^{t} = (-t, 0, 0), \quad q_4^{t} = (0, -t, 0) \qquad (t > 0).
\]
By the properties of the bisector, there exist unique numbers $0 < \vartheta_j^{t} < 1$ $(j=1,2,3,4)$ for every $t\in\R$ such that the following points lie in $B(0,z)$:
\[
p_1^{t} = (tx_1, ty_1, t\alpha_1 + \vartheta_1^{t}), \quad p_2^{t} = (tx_2, ty_2, -t\alpha_2 + \vartheta_2^{t}),
\]
\[ 
p_3^{t} = (-t, 0, \vartheta_3^{t}), \quad p_4^{t} = (0, -t, \vartheta_4^{t}).
\]
If $\lambda_1, \lambda_2, \lambda_3, \lambda_4 \geq 0$ and $\lambda_1 + \lambda_2 + \lambda_3 + \lambda_4 = 1$, then we have
\begin{equation}\label{convhull_eq}
\sum_{j=1}^4 \lambda_j p_j^{t} = \sum_{j=1}^4 \lambda_j q_j^{t} + \left(0,\,0,\, \sum_{j=1}^4 \lambda_j \vartheta_j^{t}\right)
\end{equation}
where $0 < \sum_{j=1}^4 \lambda_j \vartheta_j^{t} < 1$. 

Let us suppose that we managed to prove that the origin is in the interior of the convex hull of $\{q_j^{1}\}_{j=1}^4$. 
Then for large enough $t$-s, the set $\Conv\left(\{q_j^{t}\}_{j=1}^4\right) = \Conv\left(\{t\cdot q_j^1\}_{j=1}^4\right)$ would contain both $(0,0,-1)$ and $(0,0,1)$. 
Thus, by \eqref{convhull_eq}, $(0,0,-1+\theta)$ and $(0,0,1+\tilde\theta)$ would be in the convex hull of $\{p_j^{t}\}_{j=1}^4$ with some numbers $\theta,\tilde{\theta}\in(0,1)$. 
This, by convexity, would further imply that $(0,0,0), (0,0,1) \in \Conv(\{p_1^{t}, \dots p_4^{t}\})$, which would complete the present case. 
But the following points are clearly in the interior of the convex hull of $\{q_j^{1}\}_{j=1}^4$:
\[
\mu(\nu q_3^1 + (1-\nu) q_4^1) + (1-\mu) \left(\frac{\frac{1}{\alpha_1}}{\frac{1}{\alpha_1}+\frac{1}{\alpha_2}} q_1^1 + \frac{\frac{1}{\alpha_2}}{\frac{1}{\alpha_1}+\frac{1}{\alpha_2}} q_2^1\right)
\]
\[
= \mu(\nu q_3^1 + (1-\nu) q_4^1) + (1-\mu) \left(\frac{\frac{1}{\alpha_1}}{\frac{1}{\alpha_1}+\frac{1}{\alpha_2}} \tilde q_1^1 + \frac{\frac{1}{\alpha_2}}{\frac{1}{\alpha_1}+\frac{1}{\alpha_2}} \tilde q_2^1\right)
\]
for every $\mu,\nu\in (0,1)$, where $\tilde q_1^1 = (x_1, y_1, 0)$ and $\tilde q_2^1 = (x_2, y_2, 0)$. 
Obviously, we can find $\mu, \nu \in (0,1)$ such that the above point is exactly the origin, which completes this case.

Finally, we have to deal with the case when $d > 3$, the space $X$ is assumed to be strictly convex, and it is not an inner product space. 
We may assume without loss of generality that $\dim X = d$. 
Let us restrict ourselves into a three-dimensional subspace $Y$ for which the inherited norm is not Euclidean. 
Then by the $d=3$ case, we can find two points $x,y\in Y$, and four affine independent points in $B(x,y)\cap Y$ such that their convex hull contains $x$ and $y$. 
Since every line parallel to $\ell(x,y)$ intersects $B(x,y)$ in one point,  we can choose $d-3$ other points in $B(x,y)$ such that together with the previously mentioned four points they are still affine independent. 
Clearly, the convex hull of these $d+1$ affine independent points contains $x$ and $y$, and therefore our proof is completed.
\end{proof}

As was mentioned in the introduction, the characterization of Euclidean spaces given in \cite{KaSt} is a consequence of our results. 
In fact, we have a stronger characterization. 
We point out that the following characterization does not depend on $d$, unlike in our main results.

\begin{cor}\label{KaSt_cor}
Let $d\geq 2$. 
A $d$-dimensional normed space is Euclidean if and only if every set of $d+1$ affine independent points is a resolving set for the whole space.
\end{cor}

\begin{proof}
The necessity part is trivial. 
Concerning the sufficiency part, if $d>2$, then this is an easy consequence of Theorem \ref{main>=3d_thm} (or Theorem \ref{Blaschke_thm}). 
If $d=2$, then an easy application of Lemma \ref{bisectline_lem} completes the proof.
\end{proof}

We close this section with the following characterization of (at least three dimensional) ellipsoids. 
It follows easily from Theorem \ref{main>=3d_thm}. 

\begin{cor}\label{>=3_cor}
Let $m \geq d \geq 3$, and $K$ be a convex, compact body in $\R^m$ with non-empty interior such that $K = -K$. 
The following two conditions are equivalent:
\begin{itemize}
\item[\textup{(i)}] for every $d$ numbers $\lambda_1,\dots \lambda_d\in(0,\infty)$, and linearly independent vectors $v_1, \dots v_d \in \R^m$, the intersection 
\[
\Conv(\{0, v_1, \dots v_d\}) \cap (\partial K) \cap \left( \cap_{j=1}^d (v_j + \lambda_j \cdot \partial K) \right)
\] 
contains at most one element,
\item[\textup{(ii)}] $K$ is an ellipsoid. 
\end{itemize}
\end{cor}


\section{Final remarks, open problems}

We close this paper with some discussions. 
First, let us consider a complex normed space $(Y,\|\cdot\|)$ with dimension at least $d$ ($d\in\N$, $d>2$).
Then $Y$ can be considered as a real normed space as well with the same norm.
By the Jordan--von Neumann theorem $\|\cdot\|$ comes from a complex inner product on $Y$ if and only if $\|\cdot\|$ satisfies \eqref{par_eq}.
But this holds exactly when $\|\cdot\|$ comes from a real inner product, which is equivalent to \eqref{SRSd}.
Therefore we can obtain a characterization of complex inner product spaces of dimension at least three as well.
However in that case, one should be aware that in the definition of \eqref{SRSd} affine independence is real affine independence, and also convex hull is real convex hull.

Next, we say that the metric dimension of $X$ is $\delta \in \N$ if there exists a resolving set for $X$ with $\delta$ elements, but there is no resolving set with less elements. 
The metric dimension of a $d$-dimensional Euclidean space is clearly $d+1$. 
From our results it is clear that usually a set of $d+1$ affine independent points is not a resolving set for a $d$-dimensional normed space. 
Therefore it is reasonable to ask the following question.

\begin{problem}
Is the metric dimension of every $d$-dimensional (strictly convex) normed space less than or equal to $d+1$?
\end{problem}

The reason why we have asked "less than or equal to" is that it is not clear whether the metric dimension of a $d$-dimensional normed space can be less than $d+1$. 
However, it is quite plausible that such a phenomena cannot happen. 

Let us consider a convex, compact body $K$ in $\R^d$ with non-empty interior. 
Let us point out that the boundary $\partial K$ resolves the whole space $\R^d$ with respect to every strictly convex norm. 
The reason is that every line going through an inner point of $K$ intersects the boundary in at least two points. 
Therefore, by Lemma \ref{locstrictconv_lem}, the boundary of $K$ cannot be a subset of a bisector. 

It can be also seen that $\partial K$ resolves $K$ (but not necessarily the whole space) with respect to every (not necessarily strictly convex) norm. 
In fact, this is an easy consequence of the fact that for any two different points $x, y \in K$, the line $\ell(x,y)$ intersects $\partial K$ in at least two points, but the set $B(x,y)\cap\ell(x,y)$ has exactly one element which is $\tfrac{1}{2}(x+y)$. 

The following question is motivated by the above observations.

\begin{problem}
Find (general enough) conditions on convex compact sets of $\R^d$ such that whenever they are satisfied by a set, then the extreme points resolves this set (or the whole space) with respect to every (or every strictly convex) norm.
\end{problem}

Of course, usually $\partial K$ is much smaller than the set of all extreme points of $K$. 

In this paper we considered only distances which were induced by norms. 
Of course we may consider other natural metrics on linear spaces.

\begin{problem}
Characterize metrics in other natural classes of metrics on $\R^d$ (or on general linear spaces) such that every set of $d+1$ affine independent points is a resolving set for their convex hull. 
\end{problem}

Motivated by Corollaries \ref{2_cor} and \ref{>=3_cor}, it would be interesting to answer the following question. 

\begin{problem}
Characterize those (not necessarily centrally symmetric) convex, compact bodies $K$ in $\R^m$ ($d\leq m$) with non-empty interior such that for every $d$ numbers $\lambda_1,\dots \lambda_d\in(0,\infty)$, and linearly independent vectors $v_1, \dots v_d \in \R^m$, the intersection 
\[
\Conv(\{0, v_1, \dots v_d\}) \cap (\partial K) \cap \left( \cap_{j=1}^d (v_j + \lambda_j \cdot \partial K) \right)
\] 
has at most one element. 
\end{problem}

Several other reasonable problems could be raised concerning this direction. 
We hope that our results will inspire further investigations.


\section*{Acknowledgement}
The author emphasizes his thanks to Gerg\H{o} Nagy who posed this problem in a personal conversation.

The author was supported by the "Lend\"ulet" Program (LP2012-46/2012) of the Hungarian Academy of Sciences.


\bibliographystyle{amsplain}

\end{document}